\documentclass[11pt]{amsart}
\usepackage{amsmath,amssymb, graphicx, amscd,latexsym,comment,color}
\makeatletter
\newtheorem{Theorem}{Theorem}
\newtheorem{MainTheorem}[Theorem]{Main Theorem}
\newtheorem{Lemma}[Theorem]{Lemma}
\newtheorem{Corollary}[Theorem]{Corollary}
\newtheorem{Proposition}[Theorem]{Proposition}

\newtheorem{Remark}[Theorem]{Remark}

\newtheorem{Example}[Theorem]{Example}

\newcommand{\eps}{\varepsilon}

\newcommand\la{\lambda}
\newcommand\vphi{\varphi}

\newcommand\al{\alpha}

\newcommand\si{\sigma}
\newcommand\be{\beta}
\newcommand\Si{\Sigma}
\newcommand\ga{\gamma}
\newcommand\Ga{\Gamma}
\newcommand\de{\delta}
\newcommand\De{\Delta}

\newcommand\bfp{\mbox {\bf  p}}

\newcommand\bfx{\mbox {\bf  x}}

\newcommand\bfz{\mbox {\bf  z}}
\newcommand\bfy{\mbox {\bf  y}}
\newcommand\bfa{\mbox {\bf  a}}

\newcommand\nlind{\nl \indent}
\newcommand\nl{\newline}

\newcommand\ord{{\rm{ord}\/}}

\newcommand\Cone{\rm{Cone}\/}

\newcommand\Vol{\rm{Vol}\/}

\newcommand\pdeg{{\rm{pdeg}\/}}

\newcommand\inv{^{-1}}

\def\inv{^{-1}}

\begin{document}
\title[On the  Milnor fibration for $f(\mathbf z)\bar g(\mathbf z)$
]
{On the Milnor fibration for $f(\mathbf z)\bar g(\mathbf z)$
}

\author
[M. Oka ]
{Mutsuo Oka }
\address{\vtop{
\hbox{Department of Mathematics}
\hbox{Tokyo  University of Science}
\hbox{1-3, Kagurazaka, Shinjuku-ku}
\hbox{Tokyo 162-8601}
\hbox{\rm{E-mail}: {\rm oka@rs.kagu.tus.ac.jp}}
}}

\keywords {Mixed function, multiplicity condition}
\subjclass[2000]{14J70,14J17, 32S25}

\begin{abstract}
We consider a mixed function of type $H(\mathbf z,\bar{\mathbf z})=f(\mathbf z)\bar g({\mathbf z})$ where $f$ and $g$ are 
 convenient holomorphic functions which  have isolated critical points at the origin
 and we assume that  the intersection $f=g=0$ is  a  complete intersection variety with an isolated singularity at the origin and $H$ satisfies the multiplicity condition. 
 We will show that $H$ has a  tubular Milnor fibration at the origin. We also prove that $H$ has a spherical Milnor fibration, assuming Newton non-degeneracy of the intersection variety $f=g=0$ and Newton multiplicity condition. We give  examples which does not satisfy the Newton  multiplicity condition and which have or do not have
 Milnor fibration.
\end{abstract}
\maketitle

\maketitle
  \noindent
\section{Introduction}
Let $f(\bfz,\bar\bfz)$ be a mixed function with $f=g+i\,\,h$
where $g,h$ are real valued analytic functions of
$n$ complex variables $z_1,\dots, z_n$ or of  $2n$ real variables $\{x_j,y_j\,|\, j=1,\dots, n\}$.  Here
$\bfz=(z_1,\dots, z_n)\in \mathbb C^n$ and  $z_j=x_j+\,i \,y_j\,(j=1,\dots,n)$
with $x_j, y_j\in \mathbb R$.
The mixed hypersurface $\{f=0\}$ can be understood as the real analytic variety in $\mathbb R^{2n}$ defined by  $\{g=h=0\}$. 
$g,h$
are real valued real analytic functions of variables 
$\bfx=(x_1,\dots, x_n)$ and $\bfy=(y_1,\dots, y_n)$ but they can be considered as  mixed functions by the substitution $x_j=(z_j+\bar z_j)/2,\, y_j=-\,i (z_j-\bar z_j)/2$.
For  a mixed function $k(\mathbf z,\bar{\mathbf z})$, we use the following notations as in \cite{OkaAf}. 
\begin{eqnarray*}
d k&=&(d_{\bfx}k,d_{\bfy}k)\in \mathbb R^{2n}\,\,\text{where}\\
d_{\bfx}k&=&\left (\frac{\partial k}{\partial {x_1}},\dots, \frac{\partial k}{\partial {x_n}}\right), \,\,
d_{\bfy}k=\left (\frac{\partial k}{\partial {y_1}},\dots,\frac{\partial k}{\partial {y_n}}\right)\\
\end{eqnarray*}
$\mathbf z=\mathbf x+\,i \mathbf y\in \mathbb C^n$ and $(\mathbf x,\mathbf y)\in \mathbb R^{2n}$ are identified.
The holomorphic  gradient and the anti-holomorphic gradient  of $k$ are 
defined by
\[\begin{split}
&\partial k:=\left(\frac{\partial k}{\partial z_1},\dots, \frac{\partial k}{\partial z_n}\right ),\,
{\bar\partial} k:=\left (\frac{\partial k}{\partial \bar z_1},\dots, \frac{\partial k}{\partial\bar z_n}\right).
\end{split}
\]

Note that if $k$ is real-valued, we have the equality 
$
\overline{\partial k}={\bar\partial} k,$
and 
real gradient vector $dk\in\mathbb R^{2n}$ corresponds to the complex  gradient vector $ 2 \overline{\partial k}\in \mathbb C^n$
under the canonical correspondence
\[\mathbb R^{2n}\ni (\mathbf x,\mathbf y)\iff \mathbf z= \mathbf x+\,i \mathbf y\in \mathbb C^n.
\].

\begin{Proposition}[Proposition 1 \cite{OkaPolar}, Lemma 2 \cite{OkaAf}]\label{mixed critical}
Let $f(\mathbf z,\bar{\mathbf z})$ be a mixed function and 
put $f=g+\,i h$ as before.
The next conditions are equivalent.
\begin{enumerate}
\item $\bfa\in \mathbb C^n$ is a critical point of the mapping $f:\mathbb C^n\to \mathbb C$.
\item $dg(\bfa),\,dh(\bfa)$ are linearly dependent in $\mathbb R^{2n}$ over $\mathbb R$.
\item $\bar\partial g(\bfa,\bar \bfa),\,\bar\partial h(\bfa,\bar \bfa)$ are linearly dependent in $\mathbb C^n$ over $\mathbb R$.
\item There exists a complex number $\al$ with $|\al|=1$ such that 
$\overline{\partial f}(\bfa,\bar \bfa)=\al\,\bar \partial f(\bfa,\bar \bfa)$.
\end{enumerate}
\end{Proposition}
Under the above equivalent conditions, we say that $\bfa$ is a {\em critical point }
or  a {\em mixed singular point} of the mixed function
$f$. For brevity,  we say simply  a singular point in the sense of a mixed singular point.
\begin{Lemma} [Lemma 2, \cite{OkaAf}, cf \cite{Chen}]\label{mix-transverse}
Consider a mixed hypersurface $V_\eta=f\inv(\eta)$ and take $\bfp\in  V_\eta$.  Assume that $\bfp$ is a non-singular point of $V_\eta$
and let $k(\bfz,\bar\bfz)$ be a real valued mixed function on $\mathbb C^n$.
The following conditions are equivalent.
\begin{enumerate} 
\item The restriction $k|V_\eta$ has a critical point at  $\bfp\in V_\eta$.
\item There exists a complex number $\al$ such that 
\[
\bar\partial k(\bfp)=\al\overline{ \partial f}(\bfp,\bar\bfp)+\bar \al\bar\partial f(\bfp,\bar\bfp).
\]
\item There exist real numbers $c, d$ such that 
\[\bar\partial k(\bfp)=c\bar\partial  g(\bfp,\bar\bfp)+d\bar\partial h(\bfp,\bar\bfp).\]

\end{enumerate} 
\end{Lemma}


\section{ Fibration problem for function $f \bar g$}
\subsection{Non-degenerate mixed functions}
Let $f(\mathbf z,\bar{\mathbf z})$ be a mixed function
of $n$-variables $\mathbf z=(z_1,\dots, z_n)$. Recall that $f$ can be expanded in a convergent power series of mixed monomials
${\mathbf z}^\nu{\bar{\mathbf z}}^\mu$. In \cite{OkaMix}, we have generalized the concept of Newton boundary $\Ga(f)$
and  
defined  {\em strong non-degeneracy for  a mixed function.}
Recall that  $f$ is {\em strongly non-degenerate} if for any face $\De$ of $\Ga(f)$,
the face function $f_\De$, restricted on $\mathbb C^{*n}$ is surjective onto $\mathbb C$ and has no critical points. A convenient strongly non-degenerate function has a Milnor fibration (\cite{OkaMix}).
\subsection{Setting of our problem}
In this paper, we consider a mixed function $H$ which take the form
 $H(\mathbf z,\bar{\mathbf z})=f(\mathbf z)\bar g({\mathbf z})$ where 
$f, g$ are holomorphic functions.  Here we mean $\bar g(z):=\overline{g(z)}$.
We consider hypersurfaces  $V(f):=f\inv(0)$, $V(g):=g\inv(0)$, $V(H):=H\inv(0)$ and the intersection variety $V(f,g)=V(f)\cap V(g)$.
Note that  $H$ is  not strongly non-degenerate for $n\ge 3$ by the following reason. Suppose $H$ is non-degenerate.
Any points of the  intersection $V(f)\cap V(g)$ are singular points of $V(H)$, while a convenient strongly non-degenerate mixed function has an isolated singularity at the origin by Corollary 20 of \cite{OkaMix}. This is an obvious contradiction. Thus the mixed function $H(\mathbf z,\bar{\mathbf z})$ is far from a non-degenerate mixed function for $n\ge 3$. However $H(\mathbf z,\bar{\mathbf z})$ is  a very special type of mixed function,
as it is defined by two holomorphic functions $f,g$. The mixed hypersurface $V(H)$ is simply union of two complex analytic hypersurfaces $V(f)$ and $V(g)$ as a set but $V(g)$ is conjugate oriented by $\bar g$.
  We consider the existence of Milnor fibration for such a mixed function. Pichon and Seade have studied such functions, especially for the case $n=2$ ( \cite{Pichon-Seade0,Pichon-Seade1,Pichon-Seade2}). 
  There are also works by  Fernandez de Bobadilla and Menegon Neto \cite{JM},
   Parameswaran and Tibar \cite{Param-Tibar-revised},  Araujo dos Santos, Ribeiro and Tibar \cite{ART1},
  Araujo dos Santos, Ribeiro and Tibar \cite{ART2},  and 
  Joita and Tibar \cite{Joita-Tibar}.
Note that the link of $H$ is the union of two smooth links defined by $f$ and $g$ respectively which intersect transversely along real codimension 2 smooth variety. However  the  link of $g$ is oriented by $\bar g$.

\subsection{Basic assumption}
\subsubsection{Isolatedness}
Unless otherwise stated, we  assume  that 
\begin{enumerate}
\item[(I1)]
$f$ and $g$ are holomorphic functions
such that $V(f), V(g)$ have isolated singularity at the origin.

\item [(I2)] 
The intersection variety $V(f,g):=\{f=g=0\}$ is  a  complete intersection variety with an isolated singularity at the origin.
\end{enumerate}
We fix a positive number
$r_0>0$  so that $V(f), V(g), V(f,g)$ are only singular  at the origin in the ball $B_{r_0}^{2n}$ and for any sphere $S_r^{2n-1}$
of radius $r$ with $0<r\le r_0$ intersects transversely with these varieties.
 

 \subsubsection{Multiplicity condition}\label{multiplicity}
 There is another important  condition for $H=f\bar g$ to be fibered.
 We say that
$H$ satisfies {\em the  multiplicity-condition}  if 
 there exists a good resolution $\pi: X\to \mathbb C^n$ of the holomorphic function $h=fg$ such that
\begin{enumerate}
\item[(i)]
$\pi:X\setminus \pi\inv(\mathbf 0)\to \mathbb C^n\setminus\{\mathbf 0\}$ is biholomorphic
and the divisor defined by $\pi^*(fg)=0$ has only normal crossing  singularities and the  respective strict transforms $\tilde V(f),\,\tilde V(g)$ of $V(f)$ and $V(g)$ are smooth.

\item[(ii) ]Put $\pi\inv(\mathbf 0)=\cup_{j=1}^s D_j$ where $D_1,\dots, D_s$ are  smooth compact divisors in $X$.
Denote  the respective multiplicities of $\pi^*f$  and $\pi^*g$ along $D_j$  by $m_j$ and $n_j$.
Then $m_j\ne n_j$ for $j=1,\dots,s$.
\end{enumerate}

The multiplicity condition has been considered in the paper Fernandez de Bobadilla and Menegon Neto  \cite{JM}
for the case of plane curves.
Note that $\pi$ does not resolve completely the singularities of $V(h)$ but it resolves singularities of  $V(f)$ and $V(g)$.
\subsection{Key Lemma} 
We consider the  following  key property for the existence of the tubular Milnor fibration.

\noindent
(SN) (Isolatedness of the critical values)
{\em  There exists a positive number $r_1$  such that $0$ is 
the unique critical value of $H$ restricted on $B_{r_1}^{2n}$.}

The following lemma  shows   that (SN) condition follows from  the multiplicity condition.
\begin{Lemma}[Isolatedness of critical values]\label{nearby fiber}
Under the assumption (I1), (I2) and the multiplicity-condition,
there exist positive numbers $r_1,\,r_1\le r_0$ and $\de\ll r_1$ such that
the nearby fiber $V_\eta:=H\inv(\eta)$ has no mixed singularity in the ball
 $B_{r_1}^{2n}$ for any non-zero $\eta$ with $|\eta|\le \de$. 
\end{Lemma}
\begin{proof} We denote $\pi\inv(\mathbf 0)$  by $D$. Recall that $D=D_1\cup\cdots\cup D_s$.
For simplicity, we put $D_{s+1}=\widetilde V(f)$ and $D_{s+2}=\widetilde V(g)$.
In this notation, we put $m_{s+1}=1,m_{s+2}=0$ and $n_{s+1}=0, n_{s+2}=1$.
Take an arbitrary point $p\in D$ and assume that $p\in \bigcap_{j\in J} D_j\setminus\bigcup_{j\notin J}D_j$ where $J\subset \{1,\dots, s+2\}$. By the assumption (1),  $|J|\le n$. Then 
there is a local holomorphic chart $U_p$ with coordinates $(u_1,\dots, u_n)$ and  an injective map
$\tau:J\to \{1,\dots, n\}$ so that  $u_{\tau(j)}=0$ defines $D_j$ in $U_p$
and by the multiplicity assumption  (i) and (ii), we can write 
\begin{eqnarray}\label{pull-back-fg}
\pi^* f&=& k_f\prod_{j\in J}u_{\tau(j)}^{m_j} ,\quad
\pi^*g=k_g\prod_{j\in J}u_{\tau(j)}^{n_j}.
\end{eqnarray}
where
$k_f,k_g$ are units on $U_p$. We choose 
$U_p$ small enough so that 
$U_p\cap\bigcup_{j\notin J}D_j=\emptyset.$
Consider the pull-back  $\tilde H:=\pi^*H$.
By the assumption, we can write $\tilde H$ in $U_p$ as
\[\tilde H=k_f\bar k_g\prod_{j\in J} u_{\tau(j)}^{m_j}{\bar u_{\tau(j)}}^{n_j}.
\] 
Note that $J\cap\{1,\dots, s\}\ne \emptyset $ as $p\in D$.
Now we compute the holomorphic and anti-holomorphic gradient vectors of $\tilde H$ in $U_p$. Put 
\[
\partial {\tilde H}=(\tilde H_1,\dots, \tilde H_n),\, \,\bar\partial {\tilde H}=(\tilde H'_1,\dots, \tilde H'_n)
\]
where 
\[
\tilde H_j=\frac{\partial \tilde H}{\partial u_j}\,\,\text{and}\,\,\tilde H'_j=\frac{\partial \tilde H}{\partial\bar u_j}.
\]
Then by (\ref{pull-back-fg}), we can write
\[\begin{split}
&\tilde H_{\tau(j)}=\frac{\partial \tilde H}{\partial u_{\tau(j)}}=
u_{\tau(j)}^{m_j-1}{\bar u_{\tau(j)}}^{n_j} 
\left( m_j+u_{\tau(j)} \frac{\partial{k_f}}{\partial u_{\tau(j)}}\bar k_g\right)\prod_{k\in J,k\ne j}u_{\tau(k)}^{m_k}{\bar u_{\tau(k)}}^{n_k}\\
&\tilde H'_{\tau(j)}=\frac{\partial \tilde H}{\partial\bar u_{\tau(j)}}=
u_{\tau(j)}^{m_j}{\bar u_{\tau(j)}}^{n_j-1}
\left( n_j+\bar u_{\tau(j)}\frac{\partial{\bar{k}_g}}{\partial \bar u_{\tau(j)}} k_f\right)\prod_{k\in J,k\ne j}u_{\tau(k)}^{m_k}{\bar u_{\tau(k)}}^{n_k}
\end{split}
\]
Take  one $j\in J\cap\{1,\dots,s\}$. As $m_j\ne n_j$, we can see that 
\[\begin{split}
&|\tilde H_{\tau(j)}|\approx |u_{\tau(j)}|^{m_j+n_j-1} m_j\prod_{k\in J,k\ne j}|u_{\tau(k)}|^{m_k+n_k},\\
&|\tilde H_{\tau(j)}'|\approx |u_{\tau(j)}|^{m_j+n_j-1} n_j\prod_{k\in J,k\ne j}|u_{\tau(k)}|^{m_k+n_k}.
\end{split}
\]
Therefore $|\tilde H_{\tau(j)}/\tilde H_{\tau(j)}'|\approx m_j/n_j\ne 1$
 as  $m_j\ne n_j$ by the multiplicity condition. Thus we can take a smaller neigborhood $U_p'$ if necessary and we may assume that
$|H_{\tau(j)}|\ne |H'_{\tau(j)} |$
for any $\mathbf u\in U_p'\setminus D\cup D_{s+1}\cup D_{s+2}$. 
Therefore  by Proposition \ref{mixed critical}, 
$\tilde H: U_p'\setminus D\cup D_{s+1}\cup D_{s+2}\to \mathbb C^*$ has no critical point.
We do this operation for any $p\in D$. As $D$ is compact, we find  finite points $p_1,\dots, p_{\mu}$ such that 
$\cup_{i=1}^\mu U_{p_i} '\supset D$. Put $W=\cup_{i=1}^\mu U_{p_i} '$.  $W$ is an open set containing $D$ so that
$\tilde H:W\setminus (D\cup D_{s+1}\cup D_{s+2})\to \mathbb C^*$ has no critical point. Put ${W'}=\pi(W)$.  As $D=\pi\inv(\mathbf 0)$, ${W}'$ is an open neighborhood of the origin in $\mathbb C^n$. As 
$\pi:  W\setminus (D\cup D_{s+1}\cup D_{s+2})\to {W}'\setminus H\inv(0)$ is biholomorphic, this implies
$H: {W'}\setminus H\inv(0)\to \mathbb C^*$ has no critical point.
This proves the assertion.
\end{proof}
\begin{Remark}The multiplicity condition is a sufficient condition for the mixed smoothness of the nearby fibers but it is not always a  necessary condition.\end{Remark}
There is a paper by Parameswaran and Tibar (\cite{Param-Tibar-revised})
where they gives a  condition to characterize the isolatedness of the critical values.
This  condition for the isolatedness of the critical values  is described by a condition of $Disc(f,g)$,
which is not easy to be checked.
\subsection{Thom's $a_f$-regularity and Hamm-L\^e type Lemma}

The following follows from Lemma 3 and Corollary 4.1, \cite{Param-Tibar}.
\begin{Lemma} \label{af-property}
Assume that $f,g$ satisfy isolatedness assumption (I1) and (I2) and the multiplicity condition. Then $H$ satisfies $a_f$-regularity.
\end{Lemma}
We give a brief proof of this assertion later (\S \ref{proofofaf}).
It is well-known that $a_f$ condition implies the transversality of the nearby fibers   (Proposition 11, \cite{OkaAf}).
The following lemma corresponds  to Lemma (2.1.4), \cite{Hamm-Le1}.
 Let $r_1\le r_0$ be a small enough positive number as in Lemma  \ref{nearby fiber}.
\begin{Lemma}\label{mixed Hamm-Le}
Assume that $H$ satisfies  (I1), (I2) and the isolatedness of the critical values  (SN).
Then for any $r_2,\, r_2\le r_1$ fixed, there exists a positive number  $\de>0$ which depends on $r_2$ such that 
for any $r,\,r_2\le  r\le r_1$ and $\eta\ne 0, |\eta|\le \de$, the sphere $S_r$ 
and the nearby fiber $V({\eta}):=H\inv(\eta)$
 intersect
 transversely.
\end{Lemma}

By  Lemma \ref{mixed Hamm-Le} and Ehresman's fibration theorem \cite{Wolf1}, we have the following  tubular Milnor fibration theorem, which corresponds to Theorem 29, Theorem 52 (\cite{OkaMix}) and Theorem 9, Theorem 17 (\cite{OkaAf}).
\begin{MainTheorem} \label{tubularMilnor} 
Let $H=f\bar g$ as above and assume that $H$ satisfies the basic assumption (I1),(I2) and
the multiplicity assumption.
Let $r_1$ be as in Lemma \ref{mixed Hamm-Le}.
Take a sufficiently small $\de,\,0<\de\ll r_1$ and put 
\[
E(r_1,\de)^*:=\left\{\mathbf z\in B_{r_1}^{2n}\,|\, 0\ne |H(\mathbf z)|\le \de\right\}
\]
and $D_\de^*:=\{\eta\in \mathbb C\,|\,0\ne |\eta|\le \de\}$.
Then $H :\,E(r_1,\de)^*\to D_{\de}^*$ is a locally trivial fibration and its topological equivalence class does not depend on the choice of 
$ r_1$ and $\de$.
\end{MainTheorem}
\begin{Corollary}Assume that $H$ satisfies the isolatedness condition (I1),(I2) and the multiplicity condition. Then $H$ has a tubular Milnor fibration.
\end{Corollary}
\subsection{Spherical Milnor fibration}
We now consider the spherical Milnor mapping
$\vphi: S_{r}^{2n-1}\setminus K\to S^1$ defined by $\vphi(\mathbf z):=H(\mathbf z)/|H(\mathbf z)|$ where 
$K=V(H)\cap S_r^{2n-1}$. 
We need a stronger assumption than the basic assumption.
We  assume in this subsection that \begin{enumerate}
\item[(n1)]
$f(\bfz)$ and $g(\bfz)$ are convenient non-degenerate holomorphic  functions in the neighborhood of the origin with respect to the Newton boundaries.
\item[(n2)] $V(f,g)=\{\bfz\in \mathbb C^n\,|\, f(\mathbf z)=g(\mathbf z)=0\}$ is  a  non-degenerate complete intersection variety in the sense of Newton boundary \cite{Okabook}.
\end{enumerate}
We call (n1) and (n2) {\em the Newton non-degeneracy condition}.
\vspace{.3cm}
The hypersurfaces $V(f)$ and $V(g)$ have isolated singularities at the origin by 
the convenience and non-degeneracy assumption (n1). The intersection  variety $V(f,g)$ has also an isolated singularity at the origin and the intersections of $V(f), V(g)$ are transverse outside of the origin by (n2).  See Lemma (2.2) \cite{Okabook}.
\subsection{Newton multiplicity condition}
We further consider the following   condition. We say that
$H$ satisfies {\em the  Newton multiplicity condition}  if for
 any strictly positive weight vector $P$,  weighted degrees of $f$ and $g$ under $P$ are not equal, i.e. 
 $d(P;f)\ne d(P;g)$.

The Newton multiplicity condition can be checked by the Newton boundaries $\Ga(f)$ and $\Ga(g)$
as follows.
 \begin{Proposition} Assume that $f,g$ have   convenient Newton boundaries. Then
 $H$ satisfies Newton multiplicity condition if  and only if
 $\Ga(f)\cap \Ga(g)=\emptyset$.
\end{Proposition}
\begin{proof}
Assume that $\Ga(f)\cap \Ga(g)=\emptyset$. Then by the convenience assumption, this implies either 
\nlind
(a) $\Ga(f)$ is strictly above $\Ga(g)$ or
\nlind
 (b) $\Ga(g)$ is strictly above $\Ga(f)$.
 \nl In the case of  (a) for example, $\Ga_-(f)$ includes
$\Ga(g)$ in its interior. 
Here $\Ga_-(f)$ is the cone of $\Ga(f)$ and the origin $\mathbf 0$:
\[
\Ga_-(f):=\{t\nu\,|\, \nu\in \Ga(f),\,0\le t\le 1\}.
\]
Obviously by the convenience assumption,  (a) implies $d(P;f)>d(P;g)$ for any weight vector $P$
(respectively $(b)$ implies    $d(P;f)<d(P;g)$).

Suppose  $\Ga(f)\cap \Ga(g)\ne\emptyset$. Then we can find a hyperplane $L: a_1x_1+\dots+a_nx_n=d$ ($a_i\ge 0,\,\forall i$, $d>0$) which is tangent to 
$\Ga(f)$ and $\Ga(g)$ in the following sense. Namely $L\cap \Ga(f)\ne \emptyset,\, L\cap \Ga(g)\ne \emptyset$ and $L_+\supset \Ga(f)$ and $L_+\supset \Ga(g)$ where  
$L_+:=\{\mathbf x\in \mathbb R^n \,|\, a_1x_1+\dots+a_nx_n\ge d\}$. 
Then considering the weight vector $P=(a_1,\dots, a_n)$,
we have $d(P;f)=d(P;g)=d$.
\end{proof}

Taking an admissible toric modification $\hat\pi:X\to \mathbb C^n$ for the dual  Newton diagram $\Ga^*(fg)$,
as a good resolution, it is clear that
\begin{Proposition} Assume that the Newton non-degeneracy condition
 (n1), (n2) and Newton multiplicity condition. Then (I1),(I2) and the multiplicity condition are satisfied
\end{Proposition}
\begin{Lemma} \label{spherical}
We assume (n1),(n2) and Newton multiplicity condition. There exists a positive number $r_3$ so that $\vphi: S_r^{2n-1}\setminus K\to S^1$ has no critical points 
 for  any $r,\,0<r\le r_3$.
\end{Lemma}
\begin{proof}
 By Lemma 30 in \cite{OkaMix}, $\mathbf z\in S_r^{2n-1}\setminus K$ is a critical point of $\vphi$ if and only if 
 two vectors
$\mathbf v_2(\mathbf z)$ and $\mathbf z$ are linearly dependent over $\mathbb R$
where $\mathbf v_2(\mathbf z)=\,i \,\left(\overline{\partial \log H}(\mathbf z,\bar{\mathbf z})-\bar\partial \log H(\mathbf z,\bar{\mathbf z})\right)$.
 Assume that  the assertion does not hold.
Using the Curve Selection Lemma (\cite{Milnor,Hamm1}), we can find an analytic path $(\mathbf z(t),\la(t))\in \mathbb C^n\times \mathbb R$  for $t\in [0,1]$ such that for $t\ne 0$,
$H(\mathbf z(t),\bar{\mathbf z}(t))\ne 0$ and $\la(t)\ne 0$, $\mathbf z(0)=\mathbf 0$ and the following equality is satisfied.
\[
\,i \,\left(\overline{\partial \log H}(z(t),\bar{\mathbf z}(t))-\bar\partial \log H(z(t),\bar{\mathbf z}(t))\right)=\la(t)\mathbf z(t).
\]
Using $H=f\bar g$, this reduces to
\begin{eqnarray}
\notag \,i \left(\frac{\overline{\partial f}(\mathbf z(t))}{\bar f(\mathbf z(t))}-\frac {\overline{\partial g}(\mathbf z(t))}{\bar g(\mathbf z(t))}\right)&=&\la(t)\mathbf z(t),\,\,\text{or  equivalently}\\
\label{eq6}
\,i \left(\frac{\overline{f_j}(\mathbf z(t))}{\bar f(\mathbf z(t))}-\frac {\overline{ g_j}(\mathbf z(t))}{\bar g(\mathbf z(t))}\right)&=&\la(t)\mathbf z_j(t),\, j=1,\dots, n
\end{eqnarray}
where $f_j,g_j$ are partial derivatives. Put $m_f=\ord\,f(\mathbf z(t))$ and $m_g=\ord\,g(\mathbf z(t))$.
Consider the expansion of $\mathbf z(t)$ and $\la(t)$:
\[\begin{split}
&z_j(t)=b_jt^{p_j}+\text{(higher terms)}, \\ 
&\la(t)=\la_0t^a+\text{(higher terms)},\,a\in \mathbb Z,\, \,\la_0\in \mathbb R^*.
\end{split}
\]
Put $I=\{j|\mathbf z_j(t)\not\equiv 0\}$, $P=(p_j)_{j\in I}\in \mathbb N_+^I$  and put
\[
\ell=\min\{d(P;f^I)-m_f, d(P;g^I)-m_g)\}.
\]
To apply the non-degeneracy condition, we may assume $P$ is a weight vector of $\mathbf z$  putting $p_j$
sufficiently large for $j\notin I$ (see \cite{Okabook}).
Then we have the estimation:
\[
\begin{split}
&\ord\,\left\{ \,i \frac{\overline{f_j}(\mathbf z(t))}{\bar f(\mathbf z(t))}-\,i \frac {\overline{ g_j}(\mathbf z(t))}{\bar g(\mathbf z(t))}\right\}\ge \ell-p_j,
\end{split}
\]
where $\ord\,\la(t)\mathbf z_j(t)=a+p_j $.
If $\ell-p_j>a+p_j$ for some $j$, we get a contradiction $\la_0b_j=0$. Thus we must have 
$\ell- p_j\le a+p_j$ for any $j$.　Thus we have
\[a\ge\ell-2p_{min}\]
where  $p_{min}:=\min\{p_j|j\in I\}$.
Put 
\[\begin{split}
\eps_f&:=\begin{cases} 1,\quad &\ell=d(P;f)-m_f\\
                                          0,\quad &\ell<d(P;f)-m_f,
                                          \end{cases}\\
           \eps_g&:=\begin{cases} 1,\quad &\ell=d(P;g)-m_g\\
                                          0,\quad &\ell<d(P;g)-m_g.
                                          \end{cases}                
\end{split}
\]
\indent
Case 1. Assume that $a>\ell-2p_{min}$.  Then $a+p_j>\ell-p_j$ for any $j\in I$.

\noindent
 (a) Assume that $\eps_f\eps_g=0$.  Then by (\ref{eq6}),  we get a contradiction to the non-degeneracy assumption
$\partial f_P(\mathbf b)=0$ or $\partial g_P(\mathbf b)=0$.

\noindent
(b)
Assume  $\eps_f\eps_g=1$.  Then we get a linear relation on $\partial f_P(\mathbf b)$ and $\partial g_P(\mathbf b)$.

\indent
(b-1) Assume  $\ell<0$. Then $m_f>d(P;f),\, m_g> d(P;g)$ which implies $f_P(\mathbf b)=g_P(\mathbf b)=0$ and thus $\mathbf b\in V(f_P,g_P)$. Thus (\ref{eq6}) gives  a contradiction to the non-degeneracy of $V(f,g)$.

\indent
(b-2) If $\ell=0$, this implies $m_f=d(P;f)$ and $m_g=d(P;g)$. That is, 
$f_P(\mathbf b)\ne 0,g_P(\mathbf b)\ne 0$
and we have equality: 
\[
\frac{\overline{f_{P,j}}(\mathbf b)}{\bar f_P(\mathbf b)}-\frac {\overline{ g_{P,j}}(\mathbf b)}{\bar g_P(\mathbf b)}=0.
\]
Multiplying $p_j\bar b_j$ and adding for $j\in I$, using Euler equality
we get the equality
\[
d(P;f)
-d(P;g)
=0
\]
which is a contradiction to the Newton multiplicity-condition of $H$.

\indent
Case 2. Assume that $a=\ell-2p_{min}$. Put $J=\{j\in I| p_j=p_{min}\}$.
Put $\be_f,\be_g$ be the leading coefficients of $f(\mathbf z(t))$ and $g(\mathbf z(t))$ respectively.
Then by (\ref{eq6}), we get
\begin{eqnarray}\label{eq7}
\,i \left(
\eps_f \frac{\bar f_{P,j}(\mathbf b)}{\overline{\be_f}}-\eps_g\frac{\bar g_{P,j}(\mathbf b)}{\overline{\be_g}}\right)
=\begin{cases} \la_0 b_j,\quad &j\in J\\
0,\quad& j\notin J.
\end{cases}
\end{eqnarray}
Now we consider the differential of  $\Re(i\log f(\mathbf z(t))\bar g(\mathbf z(t)))$.
Put $P\mathbf b=(p_1b_1,\dots, p_nb_n)$.
\begin{eqnarray*}
&&\Re\left(-\frac d{dt}i\log f(\mathbf z(t))\bar g(\mathbf z(t)))\right)\\
&=&\Re\left(-\sum_{j}
i \left (\frac {1}{\be_f} f_{P,j} (\mathbf b)\eps_f b_jp_j+\frac 1{\overline{\be_g}}\overline{g_{P,j}}(\mathbf b)\eps_g \overline{b_j}p_j\right)\right)
t^{\ell-1}+\text{(higher terms)}\\
&=&\Re\left( (P\mathbf b,i\eps_f \frac{\overline{f_{P}}(\mathbf b)}{\overline{\be_f}})+(\overline{P\mathbf b}, i\eps_g\frac{g_{P}(\mathbf b)}{\be_g})
\right)t^{\ell-1}+\text{(higher terms)}\\
&=&\Re\left( P\mathbf b, i\eps_f \frac{\overline{f_P}(\mathbf b)}{\overline{\be_f}}-i\eps_g \frac{\overline{g_P}(\mathbf b)}{\overline{\be_g}}
 \right) t^{\ell-1}+\text{(higher terms)}\\
 &=&\sum_{j\in J}p_{min}\la_0|b_j|^2 t^{\ell-1}+\text{(higher terms)}.
\end{eqnarray*}
This implies $\ord\, \left(-\frac d{dt}i\log f(\mathbf z(t))\bar g(\mathbf z(t)))\right)=\ell-1$. On the other hand,
we have also
\begin{eqnarray*}
&&\left(-\frac d{dt}i\log f(\mathbf z(t))\bar g(\mathbf z(t)))\right)\\
&=&-i(m_f+m_g)t^{-1}+\text{(higher terms)}.
\end{eqnarray*}
Comparing two equalities, we see that $ \ell=0$ and  taking the real part of the last equality, we get a contradiction
$\sum_{j\in J}p_{min}\la_0|b_j|^2=0$.
\end{proof}
Combining with Lemma \ref{mixed Hamm-Le}, we get the following.
\begin{Corollary}[Spherical Milnor fibration]\label{sphericalMilnor}
Assuming (n1), (n2) and Newton multiplicity condition, $\vphi:S_r^{2n-1}\setminus K\to S^1$ gives a local trivial fibration for any $r\le \min\{r_3,r_2\}$ where $r_2$ is  a positive number in Lemma \ref{mixed Hamm-Le}.
Here $K=H\inv(0)\cap S_r^{2n-1}$.
\end{Corollary}
Pichon and Seade has proved the existence of spherical Milnor fibration assuming the isolatedness of the critical value of $H$ (Theorem 1 \cite{Pichon-Seade1}).  Araujo dos Santos and M. Tibar in \cite{AT} and 
Araujo dos Santos, Ribeiro and Tibar in \cite{ART2} studied   spherical fibration  problem in more general setting.
\begin{Remark}
Assume (n1), (n2) and Newton multiplicity condition.
Two Milnor fibrations are actually equivalent.
For the proof,  we use 
two gradient vector fields  
$\mathbf v_1(\mathbf z)$ and $\mathbf v_2(\mathbf z)$  of $\Re\log H(\mathbf z)$ and 
$\Im\log H(\mathbf z)$. 
Construct a vector field $\mathcal X(\mathbf z)$ on 
$B_r^{2n}\cap \{\mathbf z\,|\, |H(\mathbf z)|\ge \de\}$  as in  \S 5.3, \cite{OkaMix} so that 
$\Re(\mathcal X(\mathbf z),\mathbf v_2(\mathbf  z))=0,\,
\Re(\mathcal X(\mathbf z),\mathbf v_1(\mathbf z))>0,\, \Re(\mathcal X(\mathbf z),\mathbf z)>0
$.
A diffeomorphism $\psi: \partial E(r,\de)^*\to S_r^{2n-1}\setminus N(K)$ which gives an equivalence of two fibrations is constructed by the integration of this vector field. Here $\partial E(r,\de)^*:=\{\mathbf z\in B_r^{2n}\,|\, |H(\mathbf z)|=\de\}$ and $N(K):=\{\mathbf z\in S_r^{2n-1}\,|\, |H(\mathbf z)|<\de\}$.
Such a vector field is called a Milnor vector field in \cite{ART1,ART2}.
\end{Remark}
\section{Topology of the Milnor fiber}
\subsection{Fundamental group of the Milnor fiber}
Let $H=f\bar g$ and $h=fg$. We assume that $H$ satisfies assumption (I1), (I2). 
Let $F_h$  and $F_H$ be the Milnor fibers of $h$ and $H$ respectively.
$F_H$ is connected (see \cite{OkaConnectivity}).
As $\{H=0\}=\{h=0\}$ as a set,
$\pi_1(B_{r}^{2n}\setminus \{H=0\})$ is abelian for $n\ge 3$ by \cite{Le-Saito}, thus it is isomorphic to $\mathbb Z^2$.
We have two different Milnor fibrations with the same ambient space:
\[h,H:E(r_1,\de)^*\to D_\de^*.
\]
Using the homotopy exact sequence of the Milnor fibrations, we conclude
\begin{Proposition} Assume that $n\ge 3$. The fundamental groups 
$\pi_1(F_h)$ and $\pi_1(F_H)$ are isomorphic to the cyclic group $\mathbb Z$.
\end{Proposition}
\subsection{Complex subspace of the tangent space of the Milnor fiber }
In general, the tangent space of a mixed hypersurface  does not have a complex structure. However in our case, we have  the following assertion.
We assume that $r>0$ is sufficiently small so that $f$ and $g$ has no critical points in $B_r^{2n}\setminus\{\mathbf 0\}$.
\begin{Proposition}
\label{complex subspace}
 Let $H=f\bar g$ be as in Theorem \ref{tubularMilnor} and consider a Milnor fiber
$V_\eta:=H\inv(\eta)\cap B_r^{2n}$ in the tubular Milnor fibration. 
For any point $\mathbf p\in V_\eta$, $T_{\mathbf p}V_\eta$ contains a complex subspace of dimension $n-2$.
\end{Proposition}
\begin{proof} Put $a=f(\mathbf p)$ and $b=g(\mathbf p)$. Then we have $\eta=a\bar b$.
Consider two hypersurfaces $V(f,a):=f\inv(a)$ and $V(g,b):=g\inv(b)$ and their complex  tangent spaces
\[
T_{\mathbf p}V(f,a),\quad T_{\mathbf p}V(g,b).
\]
They are complex subspaces of dimension $n-1$  of the ambient space $\mathbb C^n$ and they are complex perpendicular to 
the  gradient vectors $\overline{\partial f}(\mathbf p)$ and $\overline{\partial g}(\mathbf p)$. We assert
{$T_{\mathbf p}V(f,a)\cap T_{\mathbf p}V(g,b)\subset T_{\mathbf p}V_\eta$}.
In fact  take an arbitrary tangent vector $\mathbf v\in T_{\mathbf p}V(f,a)\cap T_{\mathbf p}V(g,b)$ and take a smooth curve 
$\mathbf z(t)$ with $\mathbf z(0)=\mathbf p$ and ${d\mathbf z}/{dt}(0)=\mathbf v$. 
Then 
\[\begin{split}
&\frac{dH}{dt}|_{t=0}=\sum_{j=1}^n \frac{\partial f}{\partial z_j}(\mathbf p)\frac{dz_j}{dt}(0) \bar g(\mathbf p)+
\sum_{j=1}^n f(\mathbf p)\overline{\frac{\partial g}{\partial z_j}\frac{dz_j}{dt}(0)}\\
&=\bar g(\mathbf p)(\mathbf v,\overline{\partial f}(\mathbf p))+f(\mathbf p)\overline{(\mathbf v,\overline{ \partial g}(\mathbf p))  }=0.
\end{split}
\]
This proves $T_{\mathbf p}V_\eta\supset T_{\mathbf p}V(f,a)\cap T_{\mathbf p}V(g,b)$.
In the special case that ${\partial f}(\mathbf p)$ and ${\partial g}(\mathbf p)$ are linearly dependent at $\mathbf p$,
$T_{\mathbf p}V(f,a)=T_{\mathbf p}V(g,b)$ and the entire space $T_{\mathbf p}V_\eta$ has a complex structure.

 In the general case when  $\partial f(\mathbf p)$ and $\partial g(\mathbf p)$ are linearly independent at $\mathbf p$,
  an alternative argument  can be given as follows. Intersection variety
$V(f,g;a,b):=f\inv(a)\cap g\inv(b)$ is non-singular at $\mathbf p$ and 
$V(f,g;a,b)\subset V_\eta$ is a smooth complex subvariety of dimension $(n-2)$. Thus the tangent space $T_{\mathbf p}V(f,g;a,b)$ is a $(n-2)$-dimensional complex subspace of $T_{\mathbf p}V_\eta$.
\end{proof}
\subsubsection{Proof of Lemma \ref{af-property}}\label{proofofaf}
Now we are ready to prove Lemma \ref{af-property}.
Assume that $H=f\bar g$ satisfies isolatedness condition (I1),(I2) and the multiplicity condition.
Take $r_0$ as before. 
The hypersurface $V(H)\cap B_{r_0}^{2n}$ has a canonical stratification by 4 complex analytic strata:
\[
V(f)'=V(f)\setminus V(f,g),\, V(g)'=V(g)\setminus V(f,g),\,V(f,g)'=V(f,g)\setminus \{\mathbf 0\},\,
\{\mathbf 0\}.
\]
Consider a sequence of points $p_\nu,\,\nu=1,2,\dots$ such that
 $p_\nu\to p_0\in V(H)\setminus\{\mathbf 0\}$. 
We have to show that the limit  (if it exists)  of the tangent space $T_{p_\nu}V(H,H(p_\nu))$ 
contains the tangent space of the stratum which $p_0$ belongs to.
 Note that
$H$ has no critical point on $V(f)'\cup V(g)'$. Thus the $a_f$-regularity is obvious if $p_0\in V(f)'\cup V(g)'$.
Assume that $p_0\in V(f,g)'$. Then in the neighborhood of $p_0$, $f$ and $g$ have independent gradient
vectors by the assumption (I2).
As the tangent space $T_{p_\nu}V(H,H(p_\nu))$ contains the intersection $T_{p_\nu}V(f,f(p_\nu))\cap T_{p_\nu}V(g,g(p_\nu))$
by Proposition \ref{complex subspace},
the limit  of $T_{p_\nu}V(H,H(p_\nu))$ includes $T_pV(f,g)'$. \qed
\subsection{Jacobian curve}
We consider the critical locus of the mapping $(f,g):\mathbb C^n\to \mathbb C^2$:
\[
J(f,g):=\{\mathbf z\in \mathbb C^n\,|\, \partial f(\mathbf z),\,\partial g(\mathbf z)\text{ are linearly dependent}\}
\]
to study   the topology of $V_\eta=H\inv(\eta)\cap B_r^{2n}$.
We call {\em $J(f,g)$ the Jacobian curve of $(f,g)$ or simply the Jacobian curve of $h$.}
We assume that

 {\em  ($\star$): $J$ is a one-dimensional  curve at the origin.}
 
 In the case that $g$ is a linear form, $J$ is usually called a polar curve.
Applying a suitable Morse function $\psi: V_\eta\to \mathbb R_+$ given by a square of the distance  from a generic point near the origin, we get the following assertion.
\begin{Corollary} Under the assumption ($\star$), 
the Milnor fiber $V_\eta $ of $H$ has a homotopy type of at most $n$-dimensional CW-complex.
\end{Corollary} 
\begin{proof}Consider a Morse function $\psi$ and assume that $\mathbf p\in V_\eta$ is a critical point.
Assume first that  $\partial f(\mathbf p)$ and $\partial g(\mathbf p)$ are linearly independent at $\mathbf p$.
Then $V(f,g;a,b)\subset V_\eta$  is a smooth complex subvariey of dimension  $n-2$. Therefore
the index of of the restriction of  $\psi|V(f,g;a,b)$ is at most $n-2$ by  Milnor type argument
(\cite{Milnor}, Lemma (2.4.1) \cite{Okabook}). Thus the index of $\psi$ on $V_\eta$ is at most $n-2+2=n$. Consider  now the case
that $V_\eta\cap J$ contains some critical point of $\psi$.
By the assumption on $J$, we may assume that $J\cap V_\eta$ is a finite set, taking a sufficiently small $\eta$.
Then we modify $\psi$ a little if necessary so that $\psi$ has no critical point on $J\cap V_\eta$.
Then the assertion follows from the above discussion.
\end{proof}
\subsection{ Relation of the critical curves of $H$ and the Jacobian curve}
Let $H=f\bar g$ as before and let $C(H)$ be the closure of the  critical locus of $H:\mathbb C^n\to \mathbb C$
outside of $V(H)$.
\begin{Lemma}[Lemma 2.4,\cite{Param-Tibar}]
We have the canonical inclusion
$C(H)\subset J(f,g)$.
\end{Lemma}
\begin{proof}
Assume that $\mathbf z\in C(H)\setminus V(H)$. By Proposition \ref{mixed critical},
there exists a complex number $\al$ with $|\al|=1$ so that 
$\overline{\partial H(\mathbf z)}=\alpha\bar\partial H(\mathbf z)$. This implies
\[
\overline{\partial f(\mathbf z) }g(\mathbf z)=\alpha f(\mathbf z) \overline{\partial g(\mathbf z)}
\]
which is equivalent to 
$\partial f(\mathbf z)=\partial g(\mathbf z) \beta$ with 
$\beta=\bar{\alpha}
\overline{f(\mathbf z)}/\overline{g(\mathbf z)}$.
Thus $\mathbf z\in J(f,g)$.
\end{proof}
$C(H)$ is a real analytic variety and the inclusion $C(H)\subset J(f,g)$ is generically strict. We will see later some examples.
\subsection{Resolution of $H\inv(0)$}
In this section, we assume $H$ satisfies (n1),(n2) and the Newton multiplicity condition.
Consider the Newton boundary $\Ga(h)$  of $h(\mathbf z)=f(\mathbf z)g(\mathbf z)$ which is the same as that of $H(\mathbf z,{\bar{\mathbf z}})=f(\mathbf z)\bar g(\mathbf z)$. Take a regular subdivision $\Si^*$ of 
the dual Newton diagram $\Ga^*(h)$
and consider the associated toric modification $\hat\pi: X\to \mathbb C^n$. See \cite{Okabook} for the definition.
Put $V(f):=f\inv(0)$, $V(g):=g\inv(0)$,
$V(h):=h\inv(0)$
and $V({f,g}):=V(f)\cap V(g)$. We use the same notations as in \S 5, Chapter 3, \cite{Okabook}.
Note that 
$\hat\pi:X\to \mathbb C^n$ gives a good resolution of 
$f, g$ and $h=fg$.

\begin{Theorem} Let $\widetilde V(f),\widetilde V(g)$ and $ \widetilde V(H)$ be the strict transforms of  $V(f), \,V(g)$ and
 $V(H)=V(h)$ respectively.
Then $\widetilde V(f),\widetilde V(g)$ are non-singular and intersect transversely so that 
$\widetilde V(H)$ is the union   $\widetilde V(f)\cup\widetilde V(g)$ and  $\widetilde V(f)\cap\widetilde  V(g)=\widetilde V({f,g})$.
\end{Theorem}
\subsubsection{Holomorphic product case}
Let $h(\mathbf z)=f(\mathbf z)g(\mathbf z)$ and  let $\Si^*$ be as above. Put  $\mathcal  V^+$ be the set of vertices $P$ of $\Si^*$ which are strictly positive. We may  assume that vertices  which are not strictly positive are  the standard basis $\{E_1,\dots, E_n\}$.
(Recall that vertices are the primitive generator of 1-dimensional cone of $\Si^*$ (\cite{Okabook}.)
$\mathcal S_I$ be the set of the weight vectors $P$  in $\mathbb Q^I$  such that 
$\dim\, \Delta(P;h)=|I|-1$. 
The Milnor fiber is not simply connected but the zeta function $\zeta_{h}(t)$  of the Milnor fibration of $h=fg$ can be computed in the exact same way as  (5.3.3), \cite{Okabook} using A'Campo formula \cite{AC}.
This also gives a formula for the zeta function of the monodromy of the Milnor fibration.
\begin{Theorem}
\begin{eqnarray}
\zeta_h(t)=\prod_{P\in \mathcal V^+}(1-t^{d(P;h)})^{-\chi(E'(P))}
\end{eqnarray}
where
\[
E'(P)=\left(\hat E(P)\setminus\left( \widetilde V(h)\cup\bigcup_{Q\in \mathcal V^+,Q\ne P}\hat E(Q)  \right)\right)\cap {\hat\pi}\inv(0)
\]
For the calculation of $\chi(E'(P))$, we can use the toric stratification as in Theorem (5.3).
\begin{eqnarray}
\zeta_{h}(t)=\prod_{I}\zeta_I(t),\quad \zeta_I(t)=\prod_{P\in \mathcal S_I}(1-t^{d(P;h^I)})^{-\chi(E'(P))}
\end{eqnarray}
where 
$\chi(E'(P))=\chi(E'(P;h^I))$ and it can be computed combinatorially using Newton boundary when
$\widetilde V({h^I})$ has 
 no singularities in $\hat E(P)$, that is if $\widetilde V(f^I)\cap\widetilde V(g^I)\cap \hat E(P)= \emptyset$.
\end{Theorem}

Here $\hat E(P)$ is the exceptional divisor corresponding to the weight vector $P$ as in \cite{Okabook}.
$d(P;h)$ is the minimal value of the linear function associated with $P$ on the Newton boundary.
If $E(P):=\hat E(P)\cap \tilde V(h)$ has a singularity i.e.,  if there are components of $V(f)$ and $V(g)$ which are intersecting on $\hat E(P)$, $\chi(E'(P))$ can not computed combinatorially.
We can use additive formula of the Euler characteristic using the decomposition
\[
E(P)=E(P;f)\cup E(P;g),\,E(P;f)\cap E(P;g)=E(P;f,g).
\]
The last one can be computed using Minkowski's mix-volume (\cite{Okabook}).

\subsubsection{Mixed product $f\bar g$ case} For mixed product case $H=f\bar g$, we need the Newton-multiplicity condition for $H$. Then
  the zeta function  $\zeta_H(t)$ of the Milnor fibration of $H=f\bar g$ is given as 
\begin{Theorem}\label{mix-zeta}
\begin{eqnarray}
\zeta_H(t)=\prod_{P\in \mathcal V^+}(1-t^{\pdeg(P;H)})^{-\chi(E'(P))}
\end{eqnarray}
where $\pdeg(P;H)$ is the polar degree, i.e. $\pdeg\,(P,H)=d(P;f)-d(P;g)$.
For the calculation of $\chi(E'(P))$, we can use the toric stratification as in Theorem (5.3).
\begin{eqnarray}
\zeta_{H}(t)=\prod_{I}\zeta_I(t),\quad \zeta_I(t)=\prod_{P\in \mathcal S_I}(1-t^{\pdeg(P;H^I)})^{-\chi(E'(P))}
\end{eqnarray}
where 
$\chi(E'(P))=\chi(E'(P;H^I))$ and  it can be computed using Newton boundary if
$\widetilde V(h^I)$ has 
 no singularities in $\hat E(P)$.
\end{Theorem}
The caculation of the Euler number $\chi(E'(P))$ is the same as that of Theorem 11, \cite{OkaStrong}.

\begin{Example} Let $f(\bfz)=z_1^2+z_2^2+z_3^2$ and $g(\bfz)=z_1+z_2+z_3$. As a regular fun $\Si^*$, we
can simply take $\{E_1,E_2,E_3,P\}$ with $P=(1,1,1)$.
$E_1,\,E_2,\,E_3$ are standard basis of $\mathbb Z^3\subset \mathbb Q^3$. The corresponding toric modification is nothing but the ordinary blowing-up at the origin. 
$\hat E(P)$ is the projective space $\mathbb P^2$. Consider the chart $\Cone(P,E_2,E_3)$ with coordinates 
$(u_1,u_2,u_3)$.
As $(z_1,z_2,z_3)=(u_1,u_1u_2,u_1u_3)$, $E(P;f)=\widetilde V(f)\cap \hat E(P)$ is the conic defined by 
$1+u_2^2+u_3^2=0$ and $\hat E(P;\bar g)$ is a line defined by $1+\bar u_2+\bar u_3=0$. The pull-back of the functions are
\[\begin{split}
&\pi^*f(\mathbf u)=u_1^2(1+u_2^2+u_3^2),\\
&\pi^*\bar g(\mathbf u)=\bar u_1(1+\bar u_2+\bar u_3).
\end{split}
\]
$E(P;f)^*\cap E(P;\bar g)^*=\{(0,u_2,-1-u_2)|1+u_2^2+(-1-u_2)^2=0\}$ (2 points).
$E(P;f)$ is a conic and $E(P;\bar g)$ is a projective  line $(\text{equal to}\, 1+u_2+u_3=0)$. Thus $\chi(E(P;H))=2+2-2=2$ and 
\[\chi(E'(P;H))=\chi(\mathbb P^2)-\chi(E(P;f)\cup E(P;\bar g))
=3-2=1.\]
Thus by Theorem \ref{mix-zeta}, $\zeta_H(t)=(1-t)^{-1}$.
We know that 
\[\zeta_H(t)=\frac{P_1(t)P_3(t)}{P_0(t)P_2(t)}\]
where $P_i(t)$ is the i-th characteristic polynomial and $P_0(t)=1-t$,
$P_1(t)=1-t$ by Proposition 9.
As $H$ is a mixed homogeneous polynomial of polar degree 1, the monodromy is trivial.  $H$ defines a projective curve $C$ which is defined by $f g=0$ and
the spherical Milnor fiber is diffeomorphic to $\mathbb P^2-C$ (\cite{MC}).
 Thus $P_2(t)=(1-t)$ and $P_3(t)=1$. That is, $H_1(F)=\mathbb Z$ and $H_2(F)=\mathbb Z$
where $F$ is the Milnor fiber. 
\end{Example}

\section{ Plane curves}
In this section, we consider plane curves.
We assume (n1),(n2) and  the Newton multiplicity condition in this chapter.
Assume that  $C_f: f(x,y)=0$ and $C_g: g(x,y)=0$ are plane curves defined by  holomorphic functions $f,g$ which have convenient  non-degenerate Newton boundaries.  We  note that $h=fg$ is also Newton non-degenerate, as $C_f$ and $C_g$ do not intersect outside of the origin. This follows from the non-degeneracy assumption (n2)  of $f=g=0$.  The Newton non-degeneracy of $f=g=0$ is equivalent to the following. For any weight vector $P$ such that $\De(P;f)$ and $\De(P;g)$ are simultaneously edges of $\Ga(f)$ and $\Ga(g)$,
the face functions $f_P$ and $g_P$ do not have any common non-nomomial factor in $\mathbb C[x,y]$.
Let $\{P_1,\dots, P_r\}$ be the weight vectors corresponding to 1-faces of  $\Ga(f)$ and let $\{Q_1,\dots, Q_s\}$ be those corresponding to 1-faces of $\Ga(g)$.
Let us consider a toric modification associated with a regular fan
with weight vectors 
$\{E_1,R_1,\dots, R_a,E_2\}$  with $E_1=(1,0), E_2=(0,1)$ which is a subdivision of the union $\{P_1,\dots,P_r\}\cup \{Q_1,\dots, Q_s\}$
and let $\hat\pi:X\to\mathbb C^2$ be the corresponding toric modification.
\subsection{A'Campo's formula}
First applying  the formula by A'Campo to the resolution $\hat\pi: X\to\mathbb C^2$, 
the zeta functions $\zeta_f(\tau)$, $\zeta_g(\tau)$  and $\zeta_h(\tau)$ of $f, g$ and $h$ respectively are given as follows.
\[\begin{split}
\zeta_f(\tau)&=(1-\tau^{a_x})(1-\tau^{a_y})\prod_{j=1}^a (1-\tau^{d(R_j;f)})^{\ell_j}\\
\zeta_g(\tau)&=(1-\tau^{b_x})(1-\tau^{b_y})\prod_{j=1}^a (1-\tau^{d(R_j;g)})^{m_j}\\
\zeta_h(\tau)&=(1-\tau^{a_x+b_x})(1-\tau^{a_y+b_y})\prod_{j=1}^a (1-\tau^{d(R_j;f)+d(R_j;g)})^{\ell_j+m_j}
\end{split}
\]
where $\ell_j$  (resp. $m_j)$) is the number of irreducible factors of $f_{R_j}(x,y)$ (resp. of $g_{R_j}(x,y)$)
and $a_x, a_y$ (resp. $b_x, b_y$) are the length of $x$-axis and $y$-axis cut by $\Ga(f)$ (resp. of $\Ga(g)$).
Note that  $\ell_j=0$ or $m_j=0$  if $\dim\,\De(R_j;f)=0$ or $\dim\,\De(R_j;g)=0$ respectively.
Geometrically, $\ell_j$ and $m_j$ are the number of irreducible components of the strict transforms of $C_f$ and $C_g$ which intersect the exceptional divisor
$\hat E(R_j)$. Here we use the same notation as in \cite{Okabook}.
Note that $d(P_j;f)\ell_j$ (resp. $d(Q_k;g)m_k$)  is equal to $2\Vol\,\Cone(\De(P_j;f),\mathbf 0)$
(resp. $2\Vol\, \Cone(\De(Q_k;g),\mathbf 0)$).
The Milnor numbers of $f$ and $g$ are given by
$-\deg\,\zeta_f(\tau)+1$ and $-\deg\,\zeta_g(\tau)+1$ respectively and they are equal to the Newton numbers of $\Ga_-(f)$ and $\Ga_-(g)$ respectively (\cite{Kouchnirenko}).

Now we consider the mixed function 
$H(\mathbf z,\bar{\mathbf z}):=f(\mathbf z)\bar g({\mathbf z})$.
Consider a toric chart
$\Cone (R_j,R_{j+1})$ with coordinate chart $(u,v)$ where $u=0$  (respectively $v=0$) defines the exceptional divisor  $\hat E(R_j)$ (resp. $\hat E(R_{j+1})$)
in the notation of \S 4, Chapter 3, \cite{Okabook}. The pull back of $f,g,H$ takes the form
\[\begin{split}
&{\hat\pi}^*f=u^{d(R_j;f)} f'(u,v),\,f'(0,v)\ne 0\\
&{\hat\pi}^*g=u^{d(R_j;g)} g'(u,v),\,g'(0,v)\ne 0\\
&\pi^*H(u,v,\bar u,\bar v)=u^{d(R_j;f)}{\bar u}^{d(R_j;g)}  H'(u,v,\bar u,\bar v)
\end{split}
\]
where  $H'(u,v,\bar u,\bar v)=f'(u,v)\overline{g'}(\bar u,\bar v)$.
Note that
$ H'$ is non-zero on $\hat E(R_j)\setminus V(f,g)$. Thus if $d(R_j;f)\ne d(R_j;g)$,
$\pi^*h$ is locally topologically equivalent to the rotation around the axis $\hat E(R_j)$ by the monomial $u^{d(R_j;f)-d(R_j;g)}$.  See Lemma 12, \cite{OkaStrong}. Using the same argument as in \cite{Okabook}, we get
\begin{Theorem} [Lemma 4.2 \cite{Pichon-Seade1}, Theorem 5.4 \cite{JM}] Assume that $f,g$ are non-degenerate holomorphic functions as above and assume that they satisfy Newton multiplicity condition.
Then $H(\mathbf z,\bar{\mathbf z})$ has a Milnor fibration and the zeta function $\zeta_H(\tau)$ is given as 
\[ 
\zeta_H(\tau)
=(1-\tau^{a_x-b_x})(1-\tau^{a_y-b_y})\prod_{j=1}^a\left(1-\tau^{d(R_j;f)-d(R_j;g)}\right )^{\ell_j+m_j}.
\]
assuming $\Ga_-(f)\supset \Ga(g)$. If $\Ga_-(f)\subset \Ga(g)$, the formula is changed as
\[ 
\zeta_H(t)=(1-\tau^{b_x-a_x})(1-\tau^{b_y-a_y})\prod_{j=1}^a\left(1-\tau^{d(R_j;g)-d(R_j;f)}\right)^{\ell_j+m_j}.
\]
\end{Theorem}
\begin{Remark} Pichon and Seade have done   interesting works for $H=f\bar g$ with $n=2$ which is   not necessarily non-degenerate from Seifert graph point of view in \cite{Pichon-Seade0}. The formula of the zeta function is also obtained by Fernandez de Bobadilla and Menegon Neto
from the boundary of the Milnor fibre point of view.
\end{Remark} 
\section{Non-existence of Milnor fibration}
Let $f(\mathbf z,\bar{\mathbf z})$ be a given mixed function with $f(\mathbf 0)=0$. Here we consider again  in general dimension $n$.
For the existence of tubular Milnor fibration in the ball $B_r^{2n}$ and the independence of the isomorphism class of the fibration
by the choice of $r$, we use the following {\em $d$-regularity in \cite{CSS} or {\em  $\rho$-regularity}  condition (\cite{SCT1,SCT2,ART1}). \nl
{\em  For any fixed $r',\,0<r'\le r$, there exists a positive number $\de$ such that  for any $\eta\ne 0 ,\, |\eta|\le \de$
 the fiber $f\inv (\eta)$ is non-singular in $B_r^{2n}$ and intersects transversely with
 the sphere $S_\ga^{2n-1}$ for any $\ga,\,r'\le \ga\le r$.}
 
\subsection{Non-constant critical curve}
Consider a real curve $\si:[0,1]\to \mathbb C^n$  with $\si(0)=\mathbf 0$ such that 
any point  $\si(t)$, $0\le t\le 1$ is a critical point of $f$.
We say that $\si$ is {\em a non-constant critical curve for $f$} if $\si([0,1])\not \subset f\inv(0)$.
Namely the value of $f$ is not constantly zero along $\si$.
An obvious observation is:
\begin{Proposition}
Assume that $f$ has a non-constant critical curve. Then $f$ has no
tubular Milnor fibration.
\end{Proposition}
\subsection{Is Newton multiplicity condition  necessary?}\label{necessary}
We give several examples where the Newton multiplicity condition is not satisfied and we check if there exists a critical curve or not.  We use $(x,y)$ as the coordinates of $\mathbb C^2$.
\begin{Example}\label{Ex1}
Consider the case $f=x^3+y^2,\, g=x^2+y^2$ and $H=(x^3+y^2)(\bar x^2+\bar y^2)$.
We see that $f,g$ does not satisfy Newton multiplicity condition as $d(P;f)=d(P;Q)=2$ for $P=(1,1)$.
Note that the Jacobian curve $J(f,g)$ has three components $J_1:\,x=0$, $J_2:\,y=0$ and $J_3:\, 3x-2=0$.
$J_2$ and $J_3$ are not critical curves. $J_1$ is a critical curve. In fact putting
$\omega(t)=(0,t),\,0\le t\le 1$, we have 
\begin{eqnarray*}
\overline{\partial H}(\omega(t))=g(\omega(t))\overline{\partial f}(\omega(t))&=&(0,2t^2{\bar t}),\,\\
\bar\partial H(\omega(t))=f(\omega(t))\overline{\partial g}(\omega(t))&=&(0,2t^2{\bar t}).
\end{eqnarray*}
\end{Example}
\begin{Example}\label{still-OK1}
Let $f(x,y)=x^3-y^2,\,g(x,y)=x^2-y^3$. Then $H$ does not satisfy the Newton multiplicity condition as $d(P;f)=d(P;g)=2$ for $P=(1,1)$. The Jacobian curve is given by $xy(-9xy+4)=0$ and it has two local components
at the origin.
We can see easily none of them include a critical curve for $H$. Thus $H$ has a tubular Milnor fibration. 
\end{Example}
\begin{Example} Let $f(x,y)=x(y^2+x^3)+y^4$ and $g(x,y)=y(x^2+y^3)+x^4$. Then $H$ does not satisfy the Newton multiplicity condition. Compare with previous Example \ref{still-OK1}.
The Jacobian ideal is defined by 
$J(x,y)=-3y^2x^2+4 y^5+4 x^5-8 x^4y-8 y^4x$ and it has two irreducible factors at the origin.
One of the branch is parametrized as 
\[x(t)=t^2,\,\, y(t)=\frac 23 \sqrt{3} t^3-\frac 43 t^4+\text{(higher terms)}
\]
As $\lim_{t\to 0}({f_x(x(t),y(t))}/{g_x(x(t),y(t))})(g(x(t),y(t))/f(x(t),y(t))=8/7\ne 1$, we see this branch does not contain any non-zero critical point of $H$. As $f$ and $g$ are  symmetric in $x,y$, the other branch does not have any critical point  and $H$ has a Milnor fibration.
\end{Example}


\section{Existence problem of non-constant critical curves}
In this chapter, we consider the existence or non-existence of critical curves for the plane curve case $n=2$.
For the simplicity, we use $(x,y)$ as the coordinates of $\mathbb C^2$ in this chapter.
Though we consider the case $n=2$, the argument works for general dimension with a slight modification.
\subsection{Branches of plane curves}
Let $k(x,y)$ be a germ of holomorphic functions and we consider 
the Newton boundary $\Ga(k)$. Let $\De_1,\dots, \De_\ell$ be the  edges of $\Ga(k)$ and let 
$P_i=(p_i,q_i)$ be the corresponding weight vector for $\De_i,\,i=1,\dots,\ell$.
Consider the face function
$k_{\De_i}(x,y)$. It  has a factorization  as 
\[
k_{\De_i}(x,y)=c_ix^{a_i}y^{b_i}\prod_{j=1}^{\nu_i} (y^{p_i}-\al_{ij}^{p_i}x^{q_i})^{\mu_j},\,c_i\ne 0.
\]
For each $j$, there is a branch (or branches) $C_{ij}$ which  is parametrized as 
\[
C_{ij}:\,x(t)=t^{p_ir_{ij}},\, y(t)=\al_{ij} t^{q_ir_{ij}}+\text{(higher terms)},\,\exists r_{ij}\in \mathbb N.
\]
This follows from an admissible toric modification (see\cite{Okabook}).
We say the germ $C_{ij}$ is {\em rooted at the factor $ (y^{p_i}-\al_{ij}^{p_i}x^{q_i})^{\mu_j}$ on  the face $\De_i$}.
Every branch of $k(x,y)=0$ is rooted at  some $\De_i$ and  some $1\le j\le \nu_i$ as above except possibly the coordinate axis
$x=0$ or $y=0$ will be a branch if  $x|k$ or $y|k$ respectively.
\subsection{Branches of Jacobian curves}
Now we consider again holomorphic function $h=fg$ and the mixed function $H=f\bar g$ as before.
Consider the Jacobian curve
$J=J(f,g)$ which is  defined by  
$J(x,y)=0$ where
\[
J(x,y)=\frac{\partial f}{\partial x}(x,y)\frac{\partial g}{\partial y}(x,y)-
\frac{\partial f}{\partial y}(x,y)\frac{\partial g}{\partial x}(x,y). 
\]
Consider a face $\De_i$ of $\Ga(J)$ with weight vector $P_i=(p_i,q_i)$.
We say $\De_i$ is {\em a face of the  first type}  if $J_{P_i}(x,y)=J(f_{P_i},g_{P_i})(x,y)$.
In this case, we have 
$d(P_i;J)=d(P_i;f)+d(P_i;g)-(p_i+q_i)$.
Otherwise,  we say $\De_i$ {\em  a  hidden face}. In this latter case, we have
$J(f_{P_i},g_{P_i})=0$ and $d(P_i;J)>d(P_i;f)+d(P_i;g)-(p_i+q_i)$.

Consider a face $\De_i$ of the first kind  as above and  the factorization of the form
\[
J_{P_i}(x,y)=c_ix^{a_i} y^{b_i} \prod_{j=1}^{\nu_i}(y^{p_i}-\al_{ij}^{p_i}x^{q_i})^{\mu_j}.
\]
Consider a branch $\ga$ which is rooted at the factor $(y^{p_i}-\al_{ij}^{p_i}x^{q_i})^{\mu_j}$.
It has a parametrization for some integer $r_{ij}>0$ as follows.
\[\ga:
\begin{cases}
&x(t)=t^{p_ir_{ij}}\\
&y(t)=\al_{ij}t^{q_ir_{ij}}+\text{(higher terms)}
\end{cases},\,t\in D_\eps=\{\eta\in \mathbb C\,|\, |\eta|\le \eps\}.
\]
In the case $\mu_j>1$, it is possible that there exist  several   irreducible germs with such expression.
(In an admissible toric modification $\hat \pi:X\to \mathbb C^2$, the strict transform $\tilde \ga$ of $\ga$ intersects with  the exceptional divisor $\hat E(P_i)$ corresponding to $P_i$ (see \cite{Okabook}).
We say that $\ga$ is  {\em non-tangential} to $V(f,g)$  if
$f_{P_i}(t^{p_i},\al_{ij} t^{q_i})\ne 0$ and $g_{P_i}(t^{p_i},\al_{ij} t^{q_i})\ne 0$.
This is equivalent to $\tilde \ga\cap (\tilde V(f)\cup\tilde V(g))\cap \hat E(P_i)=\emptyset$
where $\tilde V(f),\tilde V(g)$ are  strict transforms of $V(f)$ and $V(g)$ respectively.
In particular, 
$\ga\not\subset V(f)\cup V(g)$.
\begin{Theorem}\label{existence criterion}
Assume that $\ga$ is a non-tangential  branch  of the Jacobian curve  $J$ which comes from a face of $\Ga(J)$ of first type as above. Then $\ga$ contains a non-constant critical curve of $H$ at the origin
if and only if $d(P;f)=d(P;g)$.
\end{Theorem} 
\begin{proof} Assume that $\ga(t)=(x(t),y(t))$ is a critical point of $H$ for a  sufficiently small $t$. Then   there exists a complex number 
$\la(t)$ with $|\la(t)|=1$ such that 
\begin{eqnarray}\label{crit}
\overline{\partial f(x(t),y(t))}g(x(t),y(t))=\la(t)\overline{\partial g(x(t),y(t))}f(x(t),y(t)).
\end{eqnarray}
As $\ga(t)$ is a branch of the Jacobian curve, we have the equality
\begin{eqnarray}\label{jacob}
f_x g_y-f_y g_x=0, \,\text{on }\,\,\ga(t)
\end{eqnarray} 
and taking the first lowest term, $f_{Px}g_{Py}-f_{Py}g_{Px}=0$ also  holds on $\ga(t)$.
 Note that  $f_x/g_x=f_y/g_y$ along $\ga$ and 
\[\frac{f_x((x(t),y(t) )}{g_x((x(t),y(t))}=\frac{f_{P,x}(1,\al)}{g_{P,x}(1,\al)}(1,\al)t^{d(P;f)-d(P;g)}+\text{(higher terms)}.
\]
Here we  assume $f_{P,x}(1,\al),g_{P,x}(1,\al)\ne 0$. 
(If $g_{P,x}(1,\al)=0$ for example,
it implies $f_{P,x}(1,\al)=0$ and $f_{P,y}(1,\al), g_{P,y}(1,\al)\ne 0$ by the 
equation (\ref{jacob}) and non-tangential assumption and by  the Euler equality.
In that case, we use ${f_y/g_y}$ instead of ${f_x/g_x}$.)
Put $\de:={f_{P,x}(1,\al)}/{g_{P,x}(1,\al)}$.
Then we have
\begin{eqnarray*}
\frac{f(x(t),y(t))}{g(x(t),y(t))}&=&\frac{f_P(1,\al)}{g_P(1,\al)}t^{d(P;f)-d(P;g)}+\text{(higher terms)}
\end{eqnarray*}
and using Euler equality  the coefficient of the first term can be written as 
\begin{eqnarray*}
\frac{f_P(1,\al)}{g_P(1,\al)}&=&\frac{( f_{P,x}(1,\al)+\al f_{P,y}(1,\al))/d(P;f)}
{(g_{P,x}(1,\al)+\al g_{P,y}(1,\al))/d(P;g)}\\
&=&\frac{\de (g_{P,x}(1,\al)+\al g_{P,y}(1,\al))/d(P;f) }{( g_{P,x}(1,\al)+\al g_{P,y}(1,\al))/d(P;g)}\\
&=&\de\frac{d(P;g)}{d(P;f)}.
\end{eqnarray*}
Thus we can write
\[\begin{split}
\la(t)&=
\frac{\overline{f_x(x(t),y(t))}}{\overline{g_x(x(t),y(t))}}/\frac{f(x(t),y(t))}{g(x(t),y(t))}\\
&=\frac{\bar \delta}{\delta}\frac{d(P;f)}{d(P;g)}+\text{(higher terms in $t$)}
\end{split}
\]
If $d(P;f)\ne d(P;g)$, for sufficiently small $t$, we see that 
$|\la(t)|\ne 1$ and there does not exist a critical point of $H$ by Proposition \ref{mixed critical}.
Suppose that $d(P;f)=d(P;g)$. We see that
$|\la(t)|\equiv 1$ modulo $(t)$.

If $|\la(t)|= 1$ constantly, the whole $\ga(t),\,t\in D_\eps$ is a critical curve of $H$.
This happens when $f$ and $g$ are weighted homogeneous polynomials of the same degree under the same weight $P$.

Assume that $\la(t)=a_0+a_kt^k+\text{(higher terms)},\,a_k\ne 0,\,|a_0|=1$. 
By the next Lemma \ref{2k}, there exists a positive number $\eps'$ such that for any $r\le \eps'$ fixed and $|t|=r$,
there exists $2k$ solutions $t=re^{i\theta_j},\,j=1,\dots, 2k$ of $|\la(t)|=1$.
They are parametrized real analytically in $r\in[0,\eps']$ and give non-constant critical curves for $H$.
\end{proof}
\begin{Lemma}\label{2k}
Let $\rho(t)$ be a holomorphic function on the disk $D_\eps:=\{\zeta\in\mathbb C\,|\, |\zeta|\le \eps\}$
such that $\rho(0)=a_0,\,|a_0|=1$ and $\rho(t)\not \equiv a_0$. Let $k=\ord_t\, (\rho(t)-a_0)$. Then there exists a positive number $\eps'\le \eps$ such that 
for any $0<r\le \eps'$,  there are $2k$ angles $0\le \theta_1,\dots,\theta_{2k}<2\pi$ such that 
$|\rho(re^{i\theta_j}|=1$ for $j=1,\dots, 2k$.
\end{Lemma}
\begin{proof}
Consider the Taylor expansion
\[
\rho(t)=a_0+a_kt^k+\text{(higher terms)}, \quad a_k\ne 0.
\]
Then there exists $\eps'\le \eps$ such that 
\[\begin{split}
& |a_k||t^k|\frac 12\le   |\rho(t)-a_0|\le |a_k||t^k|\frac 32,\\
 &\frac{d\arg{(\rho(re^{i\theta})-a_0)}}{d\theta}>0,
 \, \forall |t|\le \eps',
 \,\forall \theta\in [0,2\pi].
 \end{split}
\]
Thus the behavior of the loop
$\theta\mapsto \rho(re^{i\theta})-a_0$ is  topologically $k$ times rotation along the sphere of radius 
$|a_k|r^k$ centered at $a_0$. Thus  it intersects $2k$ times with the unit sphere $S_1=\{\zeta||\zeta|=1\}$.
\end{proof}
\begin{Example}
Let $f(x,y)=x^5+x^2 y^2+y^6$ and $g(x,y)=x^6+x^2y^2+y^5$.
Then $H$ does not satisfy the Newton-multiplicity condition as $d(P;f)=d(P;g)=4$ for $P=(1,1)$.
The Jacobian ideal is defined by 
\[\begin{split}
J:=&-xy j_2(x,y),\,\\
& j_2(x,y)=-10 x^5-25 x^3y^3-10 y^5+12 x^6+36 x^4 y^4+12 y^6.
\end{split}
\]
There exist three germs of Jacobian curves: $\{x=0\},\, \{y=0\}$ and 
$C=\{j_2=0\}$. It is easy to see that the first two coordinate axes are not critical curves.
$C$ consists of  
 5 smooth components. They are rooted to the face $\De$ with face function $xy(-10 x^5-10 y^5)$
 and $\De$ is a hidden face as $f_P=g_P=x^2y^2$ and $J(f_P,g_P)=0$.
They have  the Taylor expansions
\[
\begin{split}
C_1:&\quad y=-x+\frac{49}{50} x^2+\text{(higher terms)}\\
C_{2a}:&\quad y=ax+(-\frac 12+\frac{13}{50}a-\frac {13}{50}a^2+\frac 12a^3)x^2+\text{(higher terms)}\\
&\text{where}\,\, a^4-a^3+a^2-a+1=0.
\end{split}
\]
We claim each complex branch contains a critical curve for $H$.
\end{Example}
\begin{proof}
If $C_1$ or $C_{2a}$ is a critical curve, it must satisfy
$|\frac{\partial f}{\partial x}/\frac{\partial g}{\partial x}|=|f/g|$. 
Let us  see the assertion on $C_1$.
\[\begin{split}
|\frac{\partial f}{\partial x}/\frac{\partial g}{\partial x}|&=|1+\frac{49}{25}x+\text{(higher terms)}|\\
|f/g|&=|1+2x+\text{(higher terms)}|.
\end{split}
\]
Thus 
\[
|\frac{\partial f}{\partial x}/\frac{\partial g}{\partial x}|/| f/g|=|1+\frac 12 x-\text{(higher terms)}|.
\]
Consider the equation
\[
|1+\frac 12 x-\text{(higher terms)}|=1\]
on the circle $x=re^{i\theta}$.
By Lemma \ref{2k}, for fixed $r\le \eps$ small  enough, this has  two solutions
$\theta_i(r),\,i=1,2,\,0\le \theta_i(r)<2\pi$  for sufficiently small $\epsilon$.
Then  $\mathbf z_i(r):=re^{i\theta_i(r)}, i=1,2$ satisfy $\mathbf z_i(0)=\mathbf 0$
and $r\mapsto \mathbf z_i(r)$ give  non-constant critical curves for $H$. For $C_{2a}$, we omit the proof as the argument is the same.
\end{proof}
\begin{Example}
Let $f(x,y),g(x,y)$ be homogeneous polynomials of  same degree $d$. Then $J(x,y)$ is a homogeneous polynomial of degree $2d-2$. Take any branch germ $\ga$ of $J(x,y)=0$ which is non-tangential to $V(f,g)$.
Then $\ga$ contains a critical  curve for $H$.
As an example, take $f(x,y)=x^2+xy+y^2,\,g(x,y)=x^2-xy+y^2$.
Then Jacobian curve is defined by
$y^2-x^2$. This gives two branches $y=\pm x$ which are non-tangential and they are critical curves.
In this case, $C(H)=J(f,g)$.
\end{Example}
We finish this paper by  the following Lemma which follows from Lemma \ref{2k}. 

\begin{Lemma} \label{existence criterion2}
Let $f,g$ be  a holomorphic function pair without any common divisor, $H=f\bar g$ and let $J=J(f,g)$ be the Jacobian curve.
Take a local branch germ $\gamma$ of $J$ at the origin defined by a Taylor expansion  $\mathbf z(t),\,t\in D_\epsilon$ which is  not included in $ V(H)$.
$\gamma$ contains a non-constant critical curve for $H$ if and only if
\[ 
\lim_{t\to 0}\left |\frac{\frac{\partial f}{\partial z_j}(\mathbf z(t)) g(\mathbf z(t))}{\frac{\partial g}{\partial z_j}(\mathbf z(t)) f(\mathbf z(t))}\right |=1.
\]
Here $j$ is chosen so that  $\frac{\partial g}{\partial z_j}(\mathbf z(t)) \not = 0$.
\end{Lemma}

\begin{thebibliography}{10}
\bibitem{AC}
N. A'Campo.
\newblock {La fonction zeta d'une monodromie.}
\newblock {\em Commentarii Mathematici Helvetici}, 50, (1975),  233-248.

\bibitem{AT}
 R.N. Araujo dos Santos,  M. Tibar.
 Real map germs and higher open book structures.
 {\em Geom. Dedicata}, 147, (2010), 177-185.
\bibitem{SCT1}
 R.N. Araujo dos Santos, Y. Chen, M. Tibar.
Singular open book structures from real mappings, {\em Cent. Eur. J. Math.}, 11, (2013), no. 5, 817-828.

\bibitem {SCT2}
 R.N. Araujo dos Santos, Y. Chen,  M.Tibar
Real polynomial maps and singular open books at infinity, 
{\em Math. Scand. }, 118, (2016), no.1, 57-69.
\bibitem{ART1}
R.N. Araujo dos Santos, M. Ribeiro and M. Tibar.
Fibrations of highly singular map germs,
Bull. Sci. Math. 55, (2019), 92-111.
\bibitem{ART2}
R.N. Araujo dos Santos, M. Ribeiro and M. Tibar.
Milnor-Hamm sphere fibrations and the equivalence problem,
arXiv:1810.05158
\bibitem{Chen}
Y.~ Chen.
\newblock Ensembles de bifurcation des polyn\^omes mixtes et poly\`edres de Newton, Th\`ese, {\em Universit\'e de Lille I,}
2012.
\bibitem {CSS}
 J.L. Cisneros-Molina, J. Seade and J. Snoussi.
    {Milnor fibrations and the concept of {$d$}-regularity for
              analytic map germs} in 
 {\em Real and complex singularities,
   Contemp. Math.}
   {569}, 2012,
     {1-28}.
  {Amer. Math. Soc., Providence, RI}.
		
\bibitem {JM}
J. Fernandez de Bobadilla and A. Menegon Neto.
The boundary of the Milnor  fibre of complex and real analytic non-isolated singularities. 
{\em Geom Dedicata}, 173,  (2014), 143-162



\bibitem{Hamm1}
H.~Hamm.
\newblock Lokale topologische {E}igenschaften komplexer {R}\"aume.
\newblock {\em Math. Ann.},191, (1971), 235-252.

\bibitem{Hamm-Le1}
H.~A. Hamm and D.~T. L{\^e}.
\newblock Un th\'eor\`eme de {Z}ariski du type de {L}efschetz.
\newblock {\em Ann. Sci. \'Ecole Norm. Sup.}(4), (1973), 6:317-355.
\bibitem{Joita-Tibar}
C. Joita and M. Tibar.
Images of analytic map germs,
arXiv:1810.05158
\bibitem{Kouchnirenko}
{A. G. Kouchnirenko},
\newblock { Poly{\`e}dres de Newton et nombres de {Milnor}}.
\newblock{\em Invent. Math.}, 32, (1976),
{1-31}.
\bibitem{Le-Saito}
D.T. L\^e and K. Saito.
The local $\pi_1$ of the complement of a hypersurface with normal crossings in codimension 1 is abelian.
{\em Ark. Mat.}, 22,  (1984), no. 1, 1-24.

\bibitem{Milnor}
J.~Milnor.
\newblock { Singular points of complex hypersurfaces}.
\newblock {\em  Annals of Mathematics Studies,}  61. Princeton University Press,
  Princeton, N.J., 1968.

\bibitem{Okabook}
M.~Oka.
\newblock { Non-degenerate complete intersection singularity}.
\newblock {\em Hermann}, Paris, 1997.

\bibitem{OkaPolar}
M.~Oka.
\newblock Topology of polar weighted homogeneous hypersurfaces.
\newblock {\em Kodai Math. J.}, 31,  (2008),  (2):163-182.
\bibitem{OkaMix}
M.~Oka.
\newblock Non-degenerate mixed functions.
\newblock {\em Kodai Math. J.}, 33, (2010),  (1):1-62.
\bibitem{OkaStrong}
M.~Oka.
\newblock Mixed functions of strongly polar weighted homogeneous face type,
  In {\em Advanced Studies in Pure Math.}, 66, (2015), 173-202.
\bibitem{MC}
M.~Oka.
\newblock {\em On mixed projective curves,}
Singularities in Geometry and Topology, {\em IRMA Lect. Math. Theor. Phys. Eur. Math. Soc., Z\"urich.}, 20, (2012),
133-147.
  \bibitem{OkaAf}
  M.~Oka.
 {On Milnor fibrations of mixed functions, $a_f$-condition and boundary stability}.
{\em Kodai J. Math.}, 38,  (2015), 581-603.

\bibitem{OkaConnectivity}
M.~Oka.
On the connectivity of Milnor fiber for mixed functions,
arXiv: 1809.00545v1.
\bibitem{Param-Tibar-revised}
A.J. Parameswaran and M. Tibar. 
Corrigendum to "Thom irregularity and Milnor tube fibrations",
Bull. Sci. Math.,153,  (2019),120-123.
\bibitem{Param-Tibar}
A.J. Parameswaran and M. Tibar. 
Thom irregularity and Milnor tube fibrations,
{\em Bull. Sci. Math.}, 143, (2018), 58-72.
\bibitem{Pichon-Seade0}
 A. Pichon and J. Seade.
  {Real singularities and open-book decompositions of the
              3-sphere},
   {\em Ann. Fac. Sci. Toulouse Math. (6)},
{12},
      (2003),
   {2},
   {245--265}.
\bibitem{Pichon-Seade1}
 A. Pichon and J. Seade.
    {Fibred multilinks and singularities {$f\overline g$}},
    {\em Math. Ann.}, 342,  (2008),
       {3},
 {487-514}.
\bibitem{Pichon-Seade2}
A. Pichon and J. Seade.
     {Milnor fibrations and the {T}hom property for maps
              {$f\overline g$}},
   {\em Journal of Singularities}, 3, (2011),
   %
      {144-150}.
  \bibitem{Pichon-Seade3}
 A. Pichon and J. Seade.
 \newblock {Erratum: {M}ilnor fibrations and the {T}hom property for maps
              {$f\overline g$}},
   {\em Journal of Singularities}, 7, (2013),
     {21--22}.
 \bibitem{R-S-V}
M.~A.~S. Ruas, J.~Seade, and A.~Verjovsky.
\newblock On real singularities with a {M}ilnor fibration.
\newblock In {\em Trends in singularities}, Trends Math., 191-213.
  Birkh\"auser, Basel, 2002.
    \bibitem{SE1}
  J. Seade.
      {On {M}ilnor's fibration theorem for real and complex
              singularities},
in  {\em Singularities in geometry and topology},
   {127--158},
  {World Sci. Publ., Hackensack, NJ},
  {2007}.

\bibitem {TibarOberwolfach}
M. ~Tibar.
 Regularity of real mappings and non-isolated singularities, in:
{\em Topology of Real Singularities and Motivic Aspects}. Abstracts from the workshop held 30 September - 6 October, 2012.
{\em Oberwolfach Rep. 9 (2012)}, no. 4, 2933-2934.
\bibitem{WhitneyElementary}
H.~Whitney.
\newblock Elementary structure of real algebraic varieties.
\newblock {\em Ann. of Math. (2)}, (1957), 66:545--556.


\bibitem{WhitneyTangent}
H.~Whitney.
Tangents to an analytic variety.
{\em Ann. of Math.} 81, (1965), 496-549.
\bibitem{Wolf1}
J.~A. Wolf.
\newblock Differentiable fibre spaces and mappings compatible with {R}iemannian
  metrics.
\newblock {\em Michigan Math. J.}, 11, (1964), 65-70,.

\end{thebibliography}
\def\cprime{$'$} \def\cprime{$'$} \def\cprime{$'$} \def\cprime{$'$}
  \def\cprime{$'$} \def\cprime{$'$} \def\cprime{$'$} \def\cprime{$'$}

\end{document}